\documentclass[onefignum,onetabnum]{siamonline250211}

\usepackage{amsfonts}
\usepackage{amssymb}
\usepackage{mathtools}
\usepackage{graphicx}
\usepackage{epstopdf}
\usepackage{xcolor}
\usepackage{float}
\usepackage{subcaption}
\usepackage{nicematrix}
\usepackage{tikz}
\usetikzlibrary{arrows.meta, positioning, shapes}
\usepackage{blkarray}
\usepackage{enumitem}
\setlist[enumerate]{leftmargin=.5in}
\setlist[itemize]{leftmargin=.5in}

\newsiamremark{remark}{Remark}
\newsiamremark{example}{Example}
\newsiamremark{question}{Question}

\headers{Tropical Degeneration of Network Games}{H. Hu, J. Wang, and M. Wang}

\title{Tropical Degenerations of Network Games:Valuation Classes and Equilibrium Coalescence\thanks{Preprint. \today.\funding{This work was funded by the National Natural Science Foundation of China (Grant No. 72243005), the National Key Research and Development Program of China (Grant No. 2020YFA0608602), the major project of philosophy and social science research in colleges and universities of Jiangsu Province (2024SJZD129), the major project of Basic Science (Natural science) research in colleges and universities of Jiangsu Province (24KJA110002), Special Science and Technology Innovation Program for Carbon Peak and Carbon Neutralization of Jiangsu Province (Grant No. BE2022612) and Open project Fund for Ministry of Education Key Laboratory of NSLSCS (202408). }}}

\author{Hangkun Hu\thanks{Ministry of Education Key Laboratory of NSLSCS, School of Mathematical Sciences, Nanjing Normal University, Nanjing 210023, China (\email{250901005@njnu.edu.cn}).}
\and Jingyi Wang\thanks{Ministry of Education Key Laboratory of NSLSCS, School of Mathematical Sciences, Nanjing Normal University, Nanjing 210023, China (\email{wjy05100@gmail.com}).}
\and Minggang Wang\thanks{Ministry of Education Key Laboratory of NSLSCS, School of Mathematical Sciences, Nanjing Normal University, Nanjing 210023, China; Department of Mathematics, Nanjing Normal University Taizhou College, Taizhou 225300, Jiangsu, China (\email{05424@njnu.edu.cn}). Corresponding author.}}

\ifpdf
\hypersetup{
  pdftitle={Tropical Degeneration of Network Games: Valuation Coalescence and Initial-Coefficient Rigidity},
  pdfauthor={Hangkun Hu, Jingyi Wang, and Minggang Wang}
}
\fi

\begin{document}

\maketitle

\begin{abstract}
A valuation-theoretic framework is developed for studying tropical
degenerations of multilinear network games.  Equilibrium conditions are
modeled by an ideal over the Puiseux field, and valuation classes and cluster
multiplicities are used to describe the organization of Puiseux equilibrium
branches under degeneration. For valuation vectors lying in the relative interiors of generator-wise
maximal tropical cells, multilinearity is shown to force a binomial reduction
of the generator-wise initial system.  The resulting binomial systems are
governed by exponent-difference graphs, strongly connected component
decompositions, and lattice indices computed via Smith normal form.  In
particular, unimodular diagonal blocks yield initial-coefficient rigidity,
whereas non-unimodular blocks give rise to torsion-type leading-coefficient
multiplicities. The generic binomial theory is complemented by a collision-normalized
cross-prism family.  In this family, \(2^L\) Puiseux equilibrium branches share
the same valuation vector and the same leading-coefficient vector.  The
corresponding limiting initial fiber is supported at a single torus point, but
this point is nonreduced with scheme-theoretic length \(2^L\).  Thus valuation
coalescence is realized as an intrinsic scheme-theoretic collision rather than
merely as a loss of higher-order terms. These local degeneration invariants are then related to the algebraic-degree
theory of network games, showing how global equilibrium counts can be refined
by valuation classes, binomial initial systems, Smith lattice data, and
nonreduced collision fibers.
\end{abstract}



\section{Introduction}

\subsection{Research Motivation}
Network games provide a powerful mathematical framework for
modeling strategic interactions across complex, interconnected systems, ranging from
multi-tier financial markets to coupled microgrid infrastructure
\cite{KearnsLittmanSingh2001,MoEtAl2012,ZhangWuZhouChengLong2018}. A recurring structural feature in these modern multiplex systems is the presence of distinct topological layers. The present paper addresses a different, but closely related, question.
When a perturbation parameter \(t\to0\) drives a network-game family toward
a singular limiting regime, the equilibrium equations form a degenerating
system over the Puiseux field.  Such degenerations may arise from weak
couplings, near-critical payoff relations, or collisions of local strategic
responses.

\subsection{Literature Context and Problem Statement}
The computation of totally mixed Nash equilibria naturally leads to systems of multilinear
indifference equations. Datta's algebraic framework relates such systems to sparse polynomial
systems and expresses the algebraic degree of a game in terms of combinatorial data
\cite{Datta2003,Datta2010}; this algebraic viewpoint is compatible with the graphical
formulation of strategic dependence in network games
\cite{KearnsLittmanSingh2001}. More broadly, Newton polytopes, mixed volumes, and the BKK
theorem provide the standard algebro-geometric language for counting isolated complex torus
solutions of sparse polynomial systems
\cite{Bernshtein1975,HuberSturmfels1995,Sturmfels2002,CoxLittleOShea2005}.

Recent work has developed algebraic game theory in several complementary directions. Abo,
Portakal, and Sodomaco study the Nash equilibrium scheme by vector-bundle methods, using
Chern-class computations, discriminants, and resultants to understand generic and non-generic
normal-form games \cite{AboPortakalSodomaco2026}. Portakal and Sturmfels introduced the
Spohn variety as the algebraic model for dependency equilibria and related its geometry to real
algebraic geometry, oriented matroids, and Bayesian networks \cite{PortakalSturmfels2022}.
Subsequent work on conditional-independence equilibria connects these equilibrium varieties
with graphical models and algebraic statistics \cite{PortakalSendra2024,PortakalSendra2025}.
Tropical and polyhedral methods have also appeared in several equilibrium-related
settings, including product-mix auctions, tropical complementarity problems associated
with Nash equilibria, and polyhedral approaches to truthful mechanisms
\cite{TranYu2019,AllamigeonGaubertMeunier2023,JoswigKlimmSpitz2025}.
Sparse tropical polynomial systems provide another adjacent direction
\cite{AkianBereauGaubert2025}. These works are complementary to the present paper:
our focus is not on a fixed equilibrium variety or complementarity formulation, but on
the valuation-theoretic organization of equilibrium branches in a degenerating network-game
family.
Computational aspects of these equilibrium concepts are also implemented in the Macaulay2
package \texttt{GameTheory} \cite{ConnellyEtAl2025GameTheory}.

In the authors' earlier work on the algebraic degree of network games, the focus is the global
counting problem for sparse multilinear network games: Datta's permanent formula is recovered
from tropical stable intersections, and the resulting degree admits structural laws such as
multiplicativity over strongly connected components and distinct growth behavior under different
multilayer coupling operations \cite{HuWangWang2026a}. In particular, that framework gives
the algebraic degree
\[
\mathcal D\bigl(\Pi_{\mathrm{cross}}^{(L)}\bigr)=2^L
\]
for the length-\(L\) two-layer cross-prism family.

The present paper addresses a different, but closely related, question.
When a perturbation parameter \(t\to0\) drives a network-game family toward
a singular limiting regime, the equilibrium equations form a degenerating
system over the Puiseux field.  Such degenerations may arise from weak
couplings, near-critical payoff relations, or collisions of local strategic
responses. The relevant
problem is no longer only how many isolated equilibria the generic system has, but how those
algebraic branches reorganize in the tropical limit. Thus the central question becomes:
\[
\textit{as } t\to 0, \textit{ how do the equilibrium branches distribute among valuation classes,}
\]
\[
\textit{and which generator-wise initial systems control their leading-order behavior?}
\]
We answer this question by studying initial degenerations of the equilibrium system over
$K=\mathbb{C}\{\!\{t\}\!\}$ and by using tropical geometry to organize the resulting valuation
classes, binomial degenerations, and strongly connected component structures.

\subsection{Main Contributions}

Our main result is the following:

\begin{theorem}\label{thm:main-intro}
Let
\[
I=\langle f_1,\dots,f_N\rangle\subset K[x_1,\dots,x_N]
\]
be the equilibrium ideal of a multilinear network game, and let
$v\in \operatorname{Trop}(V(I))$ be a valuation vector. Define the generator-wise
initial system
\[
J_v:=\langle \operatorname{in}_v(f_1),\dots,\operatorname{in}_v(f_N)\rangle
\subset \mathbb{C}[x_1,\dots,x_N].
\]
Assume that:
\begin{enumerate}
    \item[(i)] $v$ is generator-wise generic, meaning that each
          $\operatorname{in}_v(f_i)$ contains exactly two monomials;
    \item[(ii)] after SCC ordering of the exponent-difference graph associated with
          $J_v$, each diagonal block of the exponent-difference matrix is unimodular
          (i.e., lies in $GL_{m_i}(\mathbb{Z})$).
\end{enumerate}
Then the generator-wise initial system $V(J_v)\subset(\mathbb{C}^*)^N$ consists
of a single point. Consequently, all Puiseux equilibrium branches with valuation
$v$ share the same leading coefficient vector.
\end{theorem}

This result is proved in Section~\ref{sec:binomial-classification} as
Proposition~\ref{prop:initial-rigidity}.  Section~\ref{sec:rigidity} constructs a collision-normalized cross-prism
family and proves that \(2^L\) Puiseux branches coalesce to a single
nonreduced initial torus point.
Section~\ref{sec:degree} relates cluster multiplicities and cellwise
degeneration data to the algebraic-degree theory of network games.
Section~\ref{sec:discussion} discusses the economic and geometric
interpretation of collision, coalescence, and nonreduced limiting fibers.

In addition to this main theorem, we also establish:
\begin{enumerate}
\item[(1)] \textbf{Generic binomial reduction}
(Theorem~\ref{thm:binomial_reduction}): For valuation vectors in the
relative interior of generator-wise maximal cells, each initial polynomial
contains exactly two monomials.

\item[(2)] \textbf{SCC-organized combinatorial control}
(Theorem~\ref{thm:cycle_decomposition}): The exponent-difference graph's
strongly connected components determine the block structure of the
generator-wise initial system.

\item[(3)] \textbf{Binomial block classification}
(Theorem~\ref{thm:block-classification-rigorous}): The torus solution set of
a binomial initial block is classified by the Smith normal form of its
exponent-difference matrix.

\item[(4)] \textbf{Cross-prism collision and valuation coalescence}
(Proposition~\ref{prop:cross-prism-collision}): We construct a
collision-normalized cross-prism family in which \(2^L\) Puiseux branches
share the same valuation vector and leading coefficient vector, collapsing
to a single nonreduced initial torus point of length \(2^L\).
\end{enumerate}

Together, these results show that equilibrium collapse under tropical
degeneration is not arbitrary.  In generic cells it is governed by binomial
lattice data, while in singular collision cells it is governed by the
scheme-theoretic structure of nonreduced initial fibers.

\subsection{Paper Organization}

The remainder of this paper is organized as follows.
Section~\ref{sec:setup} establishes the algebraic and tropical setup,
including the formal definitions of valuation classes, coalescence, and
cellwise rigidity.
Section~\ref{sec:initial} analyzes the initial degeneration of the
equilibrium system and proves the generic binomial reduction result.
Section~\ref{sec:combinatorial} shows how cycle and SCC decompositions
control the combinatorial structure of valuation classes.
Section~\ref{sec:binomial-classification} classifies the torus solution
structure of binomial initial systems in terms of exponent-difference
lattices and Smith normal form.
Section~\ref{sec:rigidity} constructs a collision-normalized cross-prism
family and proves that \(2^L\) Puiseux branches coalesce to a single
nonreduced initial torus point.
Section~\ref{sec:degree} relates cluster multiplicities and cellwise
degeneration data to the algebraic-degree theory of network games.
Section~\ref{sec:discussion} discusses the economic and geometric
interpretation of collision, coalescence, and nonreduced limiting fibers.

\section{Algebraic and Tropical Setup}
\label{sec:setup}

This section introduces the algebraic framework used throughout the paper.
We review multilinear network games, Puiseux series, and the valuation
structure underlying tropical degeneration. We then formalize the key
concepts of valuation classes, coalescence, and cellwise rigidity that
form the foundation of our theoretical framework.

\subsection{Multilinear Network Games}

Consider a network game with $N$ players.
Each player $i$ has two strategies, and we denote by
\[
x_i \in (0,1)
\]
the probability that player $i$ selects strategy $1$.

Let $x = (x_1,\dots,x_N)$ denote the mixed strategy profile.

For each player $i$, the Nash equilibrium condition is given by the
indifference equation
\[
f_i(x) = u_i(1,x_{-i}) - u_i(0,x_{-i}) = 0,
\]
where $u_i(s_i,x_{-i})$ denotes the expected payoff.

Because the expected payoff is linear in the mixed strategies of other
players, each function $f_i$ is a multilinear polynomial of the form

\[
f_i(x) =
\sum_{S \subseteq [N]\setminus\{i\}}
c_{i,S} \prod_{j\in S} x_j .
\]

In particular, every variable appears with degree at most one.

\begin{remark}[Multilinearity assumption]\label{rem:multilinearity}
Throughout this paper, we assume multilinear payoff functions. This assumption is essential for our binomial reduction argument (Theorem~\ref{thm:binomial_reduction}), as it ensures that Newton polytopes lie within the unit hypercube. The non-multilinear case requires different techniques: Newton polytopes may have more complex geometry, and the generic initial form may involve more than two monomials even on maximal cells. We leave the extension to general polynomial games for future work.
\end{remark}

To model degenerating payoff parameters, we introduce a parameter \(t\) and
allow coefficients to depend on Puiseux powers of \(t\).  After embedding the
parameter dependence into the Puiseux coefficient field, let
\[
f_i(x)\in K[x_1,\dots,x_N]
\]
be the resulting equilibrium polynomial.

The associated equilibrium ideal is
\[
I=\langle f_1(x),\dots,f_N(x)\rangle\subset K[x_1,\dots,x_N].
\]
The corresponding equilibrium set is the algebraic torus variety
\[
V(I)\subset (K^*)^N.
\]

\subsection{Puiseux Series and Valuations}

To study the degeneration of equilibria as $t \to 0$, we work over the
field of Puiseux series, a standard coefficient field in tropical degeneration
arguments \cite{MaclaganSturmfels2015},

\[
K = \mathbb{C}\{\!\{t\}\!\}
\]

Elements of $K$ are formal series

\[
a(t) = \sum_{k=k_0}^{\infty} a_k t^{k/m},
\]

where $a_k \in \mathbb{C}$ and $m$ is a positive integer.

The field $K$ is algebraically closed and admits a natural valuation

\[
\mathrm{val}: K^* \to \mathbb{R},
\]

defined by

\[
\mathrm{val}(a(t)) = \min \{ k/m : a_k \neq 0 \}.
\]

\begin{remark}\label{rem:complex-vs-real}
We work over the complex Puiseux field $K = \mathbb{C}\{\!\{t\}\!\}$ for algebraic degree and counting purposes. This is standard in algebraic game theory and tropical geometry, as $K$ is algebraically closed and Kapranov's theorem applies directly. For real equilibria, and in particular for probability-valued equilibria, additional sign
and positivity conditions must be imposed on the lifted branches. Thus the complex setting
provides the algebraic framework for counting and degeneration, while the real setting
requires an additional positivity analysis.
\end{remark}

For a vector

\[
x=(x_1,\dots,x_N) \in (K^*)^N
\]

we define the valuation vector

\[
v=(v_1,\dots,v_N), \qquad v_i = \mathrm{val}(x_i).
\]

These valuation vectors describe the asymptotic behaviour of equilibria
in the limit $t \to 0$.

Following the initial-ideal characterization of tropical varieties
\cite{MaclaganSturmfels2015}, the tropicalization of the equilibrium variety is defined as
\[
\operatorname{Trop}(V(I))
=
\{ v \in \mathbb{R}^N :
\operatorname{in}_v(I) \text{ contains no monomial} \}.
\]
This characterization will form the basis of the tropical equilibrium
analysis developed in the following sections.

\begin{definition}[Generator-wise initial system]\label{def:generator-wise-initial-system}
For a valuation vector \(v\), we write
\[
J_v:=\langle \operatorname{in}_v(f_1),\dots,\operatorname{in}_v(f_N)\rangle
\subset \mathbb{C}[x_1,\dots,x_N]
\]
for the generator-wise initial system associated with the chosen
equilibrium equations. We reserve \(\operatorname{in}_v(I)\) for the initial
ideal of the full equilibrium ideal \(I\). In general,
\[
J_v\subseteq \operatorname{in}_v(I),
\]
and equality requires additional Gr\"obner-basis hypotheses. The binomial
and SCC analysis in this paper is carried out at the level of \(J_v\). This
is sufficient for leading-coefficient rigidity, since the leading coefficient
of any Puiseux branch with valuation \(v\) must vanish on every
\(\operatorname{in}_v(f_i)\), hence lies in \(V(J_v)\).
\end{definition}

\subsection{Valuation Classes and Coalescence}
\label{subsec:coalescence}

We now introduce the key concepts that form the foundation of our
valuation-theoretic framework.

\begin{definition}
\label{def:valuation-class}
Let $x(t), y(t) \in (K^*)^N$ be two Puiseux equilibrium branches of the network game. 
We say that they belong to the same valuation class if
\[
\operatorname{val}(x(t)) = \operatorname{val}(y(t)).
\]
\end{definition}

\begin{definition}[Coalescence]
\label{def:coalescence}
We say that \emph{coalescence} occurs at valuation $v$ if the valuation
class $\mathcal{C}_v$ contains more than one Puiseux equilibrium branch.
Equivalently, coalescence occurs when multiple algebraic branches
aggregate to the same tropical limit.
\end{definition}

\begin{definition}[Cluster multiplicity]
\label{def:cluster-multiplicity}
The \emph{cluster multiplicity} of a valuation class $\mathcal{C}_v$ is
the number of distinct Puiseux equilibrium branches it contains:
\[
m_{\mathrm{cl}}(v) := |\mathcal{C}_v|.
\]
\end{definition}

\begin{definition}[Cellwise rigidity]
\label{def:cellwise-rigidity}
Let $\sigma$ be a tropical cell or valuation regime. A degeneration is rigid at \(v\) if all Puiseux branches with valuation \(v\)
share the same leading coefficient vector.  It is rigid on a cell \(\sigma\)
if this property holds for every \(v\in\sigma\), possibly after fixing the
same generator-wise initial system on that cell. A useful sufficient
condition is that the generator-wise initial system satisfies
\[
V(J_v)\subset(\mathbb{C}^*)^N
\]
and consists of a single point for all relevant $v$ in $\sigma$.
\end{definition}

These definitions provide the language for describing how equilibrium
branches organize under tropical degeneration. The phenomena of
coalescence and cellwise rigidity will be the central objects of study in this
paper.

\begin{remark}[Terminology]\label{rem:terminology}
We introduce three core concepts: valuation class (Definition~\ref{def:valuation-class}), 
coalescence (Definition~\ref{def:coalescence}), and cluster multiplicity (Definition~\ref{def:cluster-multiplicity}).
These terms are used throughout the paper to describe the organization of Puiseux branches under tropical degeneration.
The main structural phenomenon we establish is initial-coefficient rigidity (Proposition~\ref{prop:initial-rigidity}), 
which we also refer to as cellwise rigidity when emphasizing the localization to a specific tropical cell 
(Definition~\ref{def:cellwise-rigidity}).
\end{remark}

\begin{remark}[Relation to algebraic degree]
The algebraic degree considered in \cite{HuWangWang2026a} is a global counting
invariant: for a generic sparse network game, it is computed by the permanent of
Datta's structure matrix. The present paper refines this global count in a
degenerating family. For a zero-dimensional Puiseux family with reduced generic fiber, the
generic branch count decomposes as
\[
\mathcal{D}=\sum_v m_{\mathrm{cl}}(v),
\]
where the sum runs over valuation vectors realized by Puiseux equilibrium
branches.  More generally, if branches are counted with algebraic
multiplicity, \(m_{\mathrm{cl}}(v)\) should be interpreted as the corresponding
local branch multiplicity.
\end{remark}
\section{Initial Degeneration and Binomial Reduction}
\label{sec:initial}

In this section, we establish the fundamental connection between the
algebraic equilibrium variety and its tropical counterpart. We demonstrate
that the asymptotic behavior of the network game as the perturbation
parameter $t \to 0$ is governed by initial degenerations of the equilibrium system.

The main goal of this section is to show that the generator-wise initial degeneration
of a multilinear network game admits a binomial reduction: for
generator-wise generic valuation vectors, each initial polynomial contains exactly two
monomials. This binomial structure is the combinatorial skeleton that
controls the organization of valuation classes, as we will show in
Section~\ref{sec:combinatorial}.

\subsection{The Tropical Equilibrium Variety}
Recall that the equilibrium ideal is
\[
I=\langle f_1,\dots,f_N\rangle\subset K[x_1,\dots,x_N].
\]
The totally mixed Nash equilibria correspond exactly to the points in the
algebraic variety \(V(I)\subset (K^*)^N\).

\begin{definition}[Tropical Equilibrium]
Let \( v \in \mathbb{R}^N \). If the initial ideal \( \operatorname{in}_v(I) \) contains no monomials, we call such a vector \( v \) a \emph{tropical equilibrium}. The set of all tropical equilibria coincides with the tropical variety \( \operatorname{Trop}(V(I)) \).
\end{definition}

This algebraic definition carries profound game-theoretic meaning: a valuation vector $v$ represents the asymptotic scaling behavior of the strategy variables. The condition that $\operatorname{in}_v(I)$ contains no monomials ensures that the degenerated polynomial system still admits solutions with strictly non-zero coordinates (i.e., valid probability limits within the algebraic torus).

\begin{theorem}[Tropical lifting property for network equilibria]
Let $x(t) \in (K^*)^N$ be a totally mixed Nash equilibrium of the network game, such that $f_i(x(t)) = 0$ for all $i \in \{1, \dots, N\}$. Let $v = \operatorname{val}(x(t)) \in \mathbb{R}^N$ be its valuation vector. Then $v$ is a tropical equilibrium, meaning $v \in \operatorname{Trop}(V(I))$.
Conversely, the Fundamental Theorem of Tropical Geometry implies that every
$v\in\operatorname{Trop}(V(I))$ is realized as the valuation of a point of
$V(I)\cap (K^*)^N$, after passing to an algebraically closed Puiseux extension if necessary
\cite{MaclaganSturmfels2015}.
\end{theorem}

\begin{proof}
Let $x(t) \in V(I) \subset (K^*)^N$. By the definition of the valuation on the Puiseux field, we can expand each coordinate as $x_i(t) = c_i t^{v_i} + \text{higher order terms}$, where the leading coefficient $c_i \in \mathbb{C}^*$.

Substituting this vector \( x(t) \) into any polynomial \( f \in I \), let
\[
m = \min_{a} (\operatorname{val}(c_a) + a \cdot v).
\]
Then we have:
\[
f(x_1(t), \dots, x_N(t)) = t^{m} \operatorname{in}_v(f)(c_1, \dots, c_N) + \text{higher order terms},
\]
by the definition of the initial form with respect to the weight vector \( v \).

Since $x(t)$ is an equilibrium, $f(x(t)) = 0$ identically for all $t$. This forces the lowest degree coefficient to vanish, yielding:
\[
\operatorname{in}_v(f)(c)=0.
\]
This equality holds for all $f \in I$, which implies that the vector $c = (c_1, \dots, c_N)$ is a root of the initial ideal $\operatorname{in}_v(I)$. Since $c_i \neq 0$ for all $i$, we have $c \in (\mathbb{C}^*)^N$.

In algebraic geometry, an ideal contains a monomial if and only if its variety in the algebraic torus is empty, $\operatorname{in}_v(I)$ cannot contain any monomials. Therefore, $v \in \operatorname{Trop}(V(I))$. The converse follows from the initial-ideal formulation of the Fundamental
Theorem of Tropical Geometry \cite{MaclaganSturmfels2015}.
\end{proof}
The theorem shows that equilibrium degenerations are governed by the
structure of the initial degeneration. In the next section we exploit the
multilinear structure of network games to prove that the generator-wise initial
systems simplify dramatically, leading to binomial reductions and
cycle decompositions that control the combinatorial structure of
valuation classes.
\section{Combinatorial Control via Cycle/SCC Decomposition}
\label{sec:combinatorial}

The tropical characterization of equilibrium valuations established in the
previous section applies to arbitrary polynomial systems. However,
network games possess additional algebraic structure: each equilibrium
equation is multilinear. In this section we show that this restriction
dramatically simplifies the tropical degeneration of the equilibrium
system, and that the resulting binomial system controls the combinatorial
structure of valuation classes.

Generically, the initial forms of the equilibrium equations reduce to
binomials, which in turn induces a sparse algebraic structure governed
by the dependency graph of the game. The key insight of this section is
that this binomial structure controls the combinatorial pattern
of valuation classes: the strongly connected components of the
exponent-difference graph determine how equilibrium branches cluster
together in the tropical limit.

\subsection{Generic Binomial Reduction and Valuation Class Structure}
\label{subsec:binomial}

We begin by showing that multilinearity forces the initial forms of the
equilibrium equations to be binomials for generic valuation vectors.

\begin{theorem}[Generic Binomial Reduction]\label{thm:binomial_reduction}
Let $f_i(x)=\sum_{\alpha \in \{0,1\}^N} c_{i,\alpha} x^\alpha \in K[x_1,\dots,x_N]$
be a multilinear polynomial representing the indifference equation of player $i$.
Let $v \in \mathbb{R}^N$ lie in the relative interior of a maximal
($(N-1)$-dimensional) cell of the tropical hypersurface $\operatorname{Trop}(V(f_i))$.
Then the initial form $\operatorname{in}_v(f_i)$ contains exactly two monomials.
In particular, it is a binomial.
\end{theorem}

\begin{proof}
The tropical polynomial associated with $f_i$ is
\[
F_i(v)=\min_{\alpha}\big(\operatorname{val}(c_{i,\alpha})+\alpha\cdot v\big).
\]
A vector $v$ lies in the tropical hypersurface $\operatorname{Trop}(V(f_i))$ if and only if
this minimum is attained by at least two terms; see
\cite[Section~3.1]{MaclaganSturmfels2015}.

The Newton polytope $\operatorname{Newt}(f_i)$ is contained in the unit hypercube $[0,1]^N$
because the polynomial is multilinear. Consequently, all vertices correspond to square-free
monomials.

The coefficient valuations induce a regular subdivision of the Newton polytope, and
the tropical hypersurface is dual to the codimension-one part of this subdivision; see
\cite[Sections~3.1 and~3.4]{MaclaganSturmfels2015}. Points $v$ lying in the relative
interior of maximal cells of $\operatorname{Trop}(V(f_i))$ correspond to edges of the
dual subdivision.

Therefore the corresponding face of the Newton subdivision contains exactly two vertices,
meaning that exactly two monomials achieve the minimal $v$-weighted valuation. The initial
polynomial therefore consists of precisely these two terms, and is thus a binomial.

\textbf{Note:} Kapranov's theorem is not needed for this combinatorial statement; it will
be invoked in \S\ref{subsec:cycle-decomp} for lifting tropical equilibria to algebraic ones.
\end{proof}

The set of \(v\in\operatorname{Trop}(V(f_i))\) for which
\(\operatorname{in}_v(f_i)\) has more than two terms is the union of
non-maximal cells of the tropical hypersurface.  Under a genericity
assumption on the coefficient valuations, this exceptional locus is exactly
the \((N-2)\)-skeleton; in particular, it has codimension at least one inside
the hypersurface and Lebesgue measure zero in \(\mathbb R^N\).

\begin{remark}\label{rem:binomial-implication}
The binomial reduction of Theorem~\ref{thm:binomial_reduction} is the key
mechanism governing valuation classes in generic cells.  In such cells, the
generator-wise initial system \(J_v\) is controlled by an exponent-difference
graph and by the associated exponent lattice.  Genuine collision
coalescence, where several branches share both valuation and leading
coefficient, will arise later from a non-binomial singular initial fiber.
\end{remark}

The theorem shows that the tropical degeneration of multilinear network
games generically produces a system of binomial equations.

\subsection{Binomial Structure and Combinatorial Control of Coalescence Patterns}
\label{subsec:cycle-decomp}

We now analyze the algebraic structure of the resulting binomial system.
The key point is that, in the algebraic torus, a binomial system is
governed by the exponent differences of its monomials, equivalently by the
associated exponent lattice; this is the standard torus viewpoint on binomial
systems \cite{EisenbudSturmfels1996}. This induces a natural dependency graph
whose strongly connected components control the block structure of the
generator-wise initial system.

\begin{theorem}[SCC Structure of Generator-Wise Initial Systems]
\label{thm:cycle_decomposition}
Let $v \in \operatorname{Trop}(V(I))$ be a valuation vector such that
each initial polynomial $\operatorname{in}_v(f_i)$ is a binomial.
Then the generator-wise initial equilibrium system $J_v$ can be written in the form
\[
x^{a_i} = \lambda_i x^{b_i},
\qquad i=1,\dots,N,
\]
where $a_i,b_i \in \mathbb{Z}_{\ge 0}^N$ are exponent vectors and
$\lambda_i \in \mathbb{C}^*$.

Let $A$ be the exponent-difference matrix whose $i$th row is
$a_i-b_i \in \mathbb{Z}^N$. Define the directed dependency graph
$G_v$ on the player set $\{1,\dots,N\}$ by placing an edge $i \to j$
whenever the variable $x_j$ appears in the binomial
$\operatorname{in}_v(f_i)$.

Then the combinatorial structure of $G_v$ controls the block organization of
$J_v$. More precisely, after permuting the variables and equations according
to the strongly connected components of $G_v$, the matrix $A$ becomes block
upper triangular. Here rows are indexed by equations/players and columns by variables/players,
so the same SCC ordering is applied to both row and column indices.  In particular, the binomial equations are organized into
SCC blocks, together with a hierarchical dependence determined by the partial
order on these components.
\end{theorem}

\begin{proof}
Since each $\operatorname{in}_v(f_i)$ is a binomial, it has the form
\[
c_{i,1} x^{a_i} - c_{i,2} x^{b_i} = 0,
\]
with $c_{i,1},c_{i,2}\in\mathbb{C}^*$.
Because we work in the algebraic torus $(\mathbb{C}^*)^N$, every
coordinate is nonzero, so each equation may be rewritten as
\[
x^{a_i-b_i} = \lambda_i,
\qquad
\lambda_i = \frac{c_{i,2}}{c_{i,1}} \in \mathbb{C}^*.
\]

Thus the generator-wise initial equilibrium system is encoded by the exponent
differences $a_i-b_i$, collected as the rows of the matrix $A$.

Now define the directed graph $G_v$ by declaring that there is an edge
$i \to j$ whenever the variable $x_j$ appears in the binomial equation
for player $i$, equivalently, whenever the $j$th entry of the row
$a_i-b_i$ is nonzero. Hence the sparsity pattern of $A$ is exactly the
adjacency pattern of $G_v$.

Let
\[
\mathcal{C}_1,\dots,\mathcal{C}_m
\]
be the strongly connected components of $G_v$. Collapsing each
$\mathcal{C}_r$ to a single vertex produces the condensation graph,
which is acyclic. Therefore one can order the components so that all
edges in the condensation graph point from earlier blocks to later
blocks. After permuting the variables and equations according to this
ordering, the nonzero pattern of $A$ becomes block upper triangular.

It follows that the initial binomial system is organized by the strongly
connected components of $G_v$: variables belonging to the same component
appear in a common irreducible block of the exponent-difference matrix,
whereas variables in different components interact only through the
block-triangular structure induced by the condensation graph.

Therefore the generator-wise initial system separates into binomial subsystems
supported on the strongly connected components, with possible
inter-component dependence governed by the induced partial order.
\end{proof}

\subsection{Interpretation for Network Games}

The preceding theorems reveal a general mechanism governing the tropical
degeneration of multilinear network games.

Under generator-wise generic perturbations, the equilibrium equations reduce
to sparse binomial systems whose algebraic structure is controlled by the
exponent-difference graph.  The SCC decomposition of this graph provides a
recursive organization of the initial system: after topological ordering, the
corresponding exponent-difference matrix becomes block upper triangular, and
the binomial equations can be solved along the induced partial order.

This binomial reduction and SCC decomposition provide the combinatorial
skeleton for valuation classes in generic cells.  They explain when a
valuation class has a unique leading coefficient and when torsion in the
exponent lattice produces several leading coefficient vectors.

The next section classifies these binomial blocks.  After that, Section
\ref{sec:rigidity} studies a complementary collision regime in which the
initial system is no longer generator-wise binomial and the multiplicity is
carried by a nonreduced initial fiber.

\section{Classification of Binomial Blocks}
\label{sec:binomial-classification}

In this section, we strengthen the structural results of Section~\ref{sec:combinatorial}
by classifying the solution sets of the binomial subsystems arising from tropical degeneration.
The key point is that, once a generator-wise initial system becomes binomial, its torus solution structure
is governed not merely by the underlying directed graph, but by the lattice-theoretic data
of the associated exponent-difference matrix, as in the general theory of binomial ideals
\cite{EisenbudSturmfels1996}.

\subsection{Exponent-Difference Matrices and Block Decomposition}

Let
\[
I=\langle f_1,\dots,f_N\rangle \subset K[x_1,\dots,x_N]
\]
be the equilibrium ideal of a multilinear network game over the Puiseux field
\[
K=\mathbb{C}\{\!\{t\}\!\}.
\]
By Theorem~\ref{thm:binomial_reduction}, if \(v\) is generator-wise generic,
meaning that \(v\) lies in the relative interior of a maximal cell of each relevant
tropical hypersurface \(\operatorname{Trop}(V(f_i))\), then each initial polynomial
\[
\operatorname{in}_v(f_i)
\]
is a binomial:
\[
\operatorname{in}_v(f_i)=c_{i,1}\mathbf{x}^{\alpha_i}-c_{i,2}\mathbf{x}^{\beta_i},
\qquad
c_{i,1},c_{i,2}\in \mathbb C^*,\quad
\alpha_i,\beta_i\in \{0,1\}^N.
\]

\begin{definition}[Exponent-difference matrix]
\label{def:exp-diff-matrix-rigorous}
The \emph{exponent-difference matrix} associated with the binomial
generator-wise initial system
\[
J_v=\langle \operatorname{in}_v(f_1),\dots,\operatorname{in}_v(f_N)\rangle
\]
is the integer matrix
\[
A_v\in \mathbb{Z}^{N\times N}
\]
whose $i$-th row is
\[
(A_v)_{i,\bullet}=\alpha_i-\beta_i.
\]
Equivalently, after dividing each binomial equation by one of its monomials, the generator-wise initial system
can be written in the form
\[
\mathbf{x}^{A_v}=\mathbf{c},
\qquad
\mathbf{c}\in (\mathbb C^*)^N,
\]
where the $i$-th equation is
\[
\mathbf{x}^{\alpha_i-\beta_i}=c_i.
\]
\end{definition}

\begin{remark}
The directed graph attached to $A_v$ records dependency, but the torus solution structure
is controlled by the integer lattice generated by the rows of $A_v$ (equivalently, by the image
of $A_v^T$ in $\mathbb{Z}^N$). Thus the graph-theoretic SCC decomposition is only the first layer
of structure; the decisive algebraic invariant is the lattice index of the corresponding block.
\end{remark}

Let $G_v$ be the directed graph on $\{1,\dots,N\}$ with an edge $i\to j$ whenever
$(A_v)_{ij}\neq 0$; equivalently, equation $i$ depends on variable $x_j$. Reversing
all arrows gives the transpose convention and preserves the SCCs, but reverses the
block order. Let
\[
C_1,\dots,C_k
\]
be its strongly connected components, ordered topologically. After permuting variables and equations
accordingly, $A_v$ takes block upper triangular form
\[
A_v=
\begin{pmatrix}
B_1 & * & * & \cdots & *\\
0   & B_2 & * & \cdots & *\\
0   & 0   & B_3 & \cdots & *\\
\vdots & \vdots & \vdots & \ddots & \vdots\\
0 & 0 & 0 & \cdots & B_k
\end{pmatrix},
\]
where each diagonal block $B_r$ corresponds to the SCC $C_r$.

\begin{remark}[Recursive block solving]
\label{rem:block-recursive}
The block triangular form does not mean that the diagonal blocks are completely independent.
Rather, it implies that the system can be solved recursively along the SCC partial order:
once the variables in later blocks have been fixed, the equations in the preceding block become
a binomial system with constant right-hand side. Hence the dimensions and finite multiplicities
arising at each stage are controlled by the lattice data of the diagonal blocks.
\end{remark}

\subsection{Classification of a Single Binomial Block}

We now classify the torus solution set of a single square binomial block.

\begin{theorem}[Binomial Block Classification]
\label{thm:block-classification-rigorous}
Let
\[
B\in \mathbb{Z}^{m\times m}
\]
and consider the binomial system
\[
\mathbf{x}^B=\mathbf{c},
\qquad
\mathbf{x}\in (K^*)^m,\quad \mathbf{c}\in (K^*)^m.
\]
Let
\[
\phi_B:(K^*)^m\to (K^*)^m,
\qquad
\phi_B(\mathbf{x})=\mathbf{x}^B
\]
be the associated monomial map, and let
\[
L=\operatorname{im}(B^T)\subset \mathbb{Z}^m.
\]
Assume the solution set
\[
X_B(\mathbf{c})=\{\mathbf{x}\in (K^*)^m:\mathbf{x}^B=\mathbf{c}\}
\]
is nonempty. Then:

\begin{enumerate}
\item If $\operatorname{rank}(B)=m$, then $X_B(\mathbf{c})$ is a torsor under the finite group
\[
\ker(\phi_B),
\]
and in particular
\[
|X_B(\mathbf{c})|
=
|\ker(\phi_B)|
=
[\mathbb{Z}^m:L].
\]
Equivalently, if $B$ is invertible over $\mathbb{Q}$, then
\[
|X_B(\mathbf{c})|=|\det(B)|.
\]

\item If $\operatorname{rank}(B)=m$ and
\[
[\mathbb{Z}^m:L]=1,
\]
equivalently $B\in GL_m(\mathbb{Z})$, then $X_B(\mathbf{c})$ consists of a unique torus point.

\item If $\operatorname{rank}(B)=r<m$, then $X_B(\mathbf{c})$ is a translate of $\ker(\phi_B)$.
Moreover, the identity component of $ker(\phi_B)$ is an algebraic subtorus of dimension
\[
m-r,
\]
and $X_B(\mathbf{c})$ is therefore a torsor under a group of the form
\[
(\text{finite torsion})\times (K^*)^{m-r}.
\]
In particular, every nonempty solution set has pure dimension $m-r$.
\end{enumerate}
\end{theorem}

\begin{proof}
The proof is standard in the theory of binomial equations on algebraic tori.
For a comprehensive treatment of binomial ideals and their solution sets,
see \cite{EisenbudSturmfels1996}. We include the argument for completeness.

\smallskip
\noindent\textbf{Step 1: The monomial map viewpoint.}
The matrix $B$ defines a homomorphism of algebraic tori
\[
\phi_B:(K^*)^m\to (K^*)^m,\qquad \mathbf{x}\mapsto \mathbf{x}^B.
\]
If $X_B(\mathbf{c})$ is nonempty and $\mathbf{x}^{(0)}\in X_B(\mathbf{c})$, then every other solution
is of the form
\[
\mathbf{x}=\mathbf{x}^{(0)}\cdot \mathbf{u}
\]
with
\[
\phi_B(\mathbf{u})=\mathbf{1}.
\]
Thus
\[
X_B(\mathbf{c})=\mathbf{x}^{(0)}\cdot \ker(\phi_B),
\]
so the solution set is a torsor under $\ker(\phi_B)$.

\smallskip
\noindent\textbf{Step 2: Smith normal form.}
Choose unimodular matrices
\[
U,V\in GL_m(\mathbb{Z})
\]
such that
\[
UBV
=
S
=
\operatorname{diag}(s_1,\dots,s_r,0,\dots,0),
\qquad
s_1\mid s_2\mid \cdots \mid s_r,\quad s_i>0,
\]
where $r=\operatorname{rank}(B)$.

Because $U$ and $V$ are unimodular, they induce monomial automorphisms of $(K^*)^m$.
After applying these automorphisms to the equations and variables, the system becomes equivalent to
\[
y_1^{s_1}=d_1,\quad \dots,\quad y_r^{s_r}=d_r,
\]
together with
\[
1=d_{r+1},\quad \dots,\quad 1=d_m,
\]
for suitable constants $d_i\in K^*$ determined by $\mathbf{c}$.

\smallskip
\noindent\textbf{Step 3: Full-rank case.}
Assume $r=m$. Then the transformed system is
\[
y_i^{s_i}=d_i,\qquad i=1,\dots,m.
\]
Since $K$ is algebraically closed of characteristic $0$, each equation $y_i^{s_i}=d_i$ has exactly
$s_i$ distinct solutions in $K^*$. Therefore the full system has exactly
\[
\prod_{i=1}^m s_i
\]
solutions. By Smith normal form,
\[
\prod_{i=1}^m s_i=[\mathbb{Z}^m:L]=|\det(B)|.
\]
This proves (i).

If in addition $[\mathbb{Z}^m:L]=1$, then each $s_i=1$, so the transformed system has a unique solution.
Hence the original system also has a unique torus solution. This proves (ii).

\smallskip
\noindent\textbf{Step 4: Rank-deficient case.}
Assume $r<m$. Then the transformed system has the form
\[
y_i^{s_i}=d_i,\qquad i=1,\dots,r,
\]
together with the compatibility conditions
\[
1=d_{r+1},\dots,1=d_m.
\]
By the nonemptiness assumption, these compatibility conditions are satisfied.
For $i\le r$, each equation contributes a finite torsion factor of order $s_i$;
for the remaining $m-r$ coordinates, the variables
\[
y_{r+1},\dots,y_m
\]
are free parameters in $K^*$. Hence the solution set is a translate of
\[
\mu_{s_1}\times \cdots \times \mu_{s_r}\times (K^*)^{m-r},
\]
where $\mu_{s_i}$ denotes the group of $s_i$-th roots of unity in $K^*$.
In particular, the identity component has dimension $m-r$, and every nonempty solution set has pure dimension
$m-r$. This proves (iii).
\end{proof}

\begin{remark}
The theorem shows that the correct algebraic invariant is not merely the SCC topology,
but the lattice index
\[
[\mathbb{Z}^m:\operatorname{im}(B^T)].
\]
Graph structure helps organize the equations, but uniqueness, finite multiplicity, or positive-dimensional
freedom are determined by the Smith invariants of the block.
\end{remark}

\subsection{A Practical Unimodularity Criterion}

The previous theorem yields a clean uniqueness criterion once the block matrix is unimodular.
For applications, it is useful to isolate a simple sufficient condition that
occurs in monomial-normalized binomial blocks.

\begin{corollary}[Monomial-normalized unimodular blocks]
\label{cor:normalized-unimodular}
Suppose a binomial block can be written, after monomial normalization
and reordering of variables, in the form
\[
x_{\sigma(1)} = c_1,\quad x_{\sigma(2)} = c_2,\quad \dots,\quad x_{\sigma(m)} = c_m,
\]
for some permutation $\sigma \in S_m$ and constants $c_i \in K^*$.
Equivalently, its exponent-difference matrix is a permutation matrix.

Then the block is unimodular and hence admits a unique torus solution.
\end{corollary}

\begin{proof}
The exponent-difference matrix is a permutation matrix $P_\sigma$, so
\[
\det(P_\sigma) = \pm 1.
\]
Hence $P_\sigma \in GL_m(\mathbb{Z})$ is unimodular.
The conclusion then follows from Theorem~\ref{thm:block-classification-rigorous}(ii).
\end{proof}

\begin{remark}
This corollary is intentionally formulated at the level of the binomial equations, rather than
purely at the level of the directed graph. A directed cycle in the dependency graph does not by itself
force the associated exponent-difference matrix to be unimodular. What matters is the monomial-normalized
shape of the equations.
\end{remark}

\subsection{Initial-Coefficient Rigidity in Unimodular Cells}

We now record the rigidity consequence relevant for tropical lifting.

\begin{proposition}[Initial-coefficient rigidity]
\label{prop:initial-rigidity}
Let $v\in \operatorname{Trop}(V(I))$ be a valuation vector, and let
\[
J_v=\langle \operatorname{in}_v(f_1),\dots,\operatorname{in}_v(f_N)\rangle
\]
be the generator-wise initial system. Assume that:
\begin{enumerate}
    \item[(i)] $v$ is generator-wise generic, so that each $\operatorname{in}_v(f_i)$ is a binomial
          by Theorem~\ref{thm:binomial_reduction};
    \item[(ii)] after SCC ordering of the exponent-difference graph associated with \(J_v\),
          each diagonal block $B_1,\dots,B_r$ of the exponent-difference matrix is unimodular
          (i.e., $B_i \in GL_{m_i}(\mathbb{Z})$).
\end{enumerate}
Then:
\begin{enumerate}
    \item[(a)] The generator-wise initial system \(V(J_v)\subset(\mathbb{C}^*)^N\) consists
          of a single point \(\mathbf{c}^*\).
    \item[(b)] For any Puiseux branch $x(t) \in V(I) \subset (K^*)^N$ with $\operatorname{val}(x(t)) = v$,
          the leading coefficient satisfies $\operatorname{in}(x(t)) = \mathbf{c}^*$.
\end{enumerate}
\end{proposition}

\begin{proof}
By Theorem~\ref{thm:cycle_decomposition} (Exponent-Difference Graph Structure), after permuting variables and equations according
to the SCCs of the dependency graph $G_v$, the exponent-difference matrix $A$ takes
block upper triangular form:
\[
A = \begin{pmatrix}
B_1 & A_{12} & \cdots & A_{1r} \\
0 & B_2 & \cdots & A_{2r} \\
\vdots & \vdots & \ddots & \vdots \\
0 & 0 & \cdots & B_r
\end{pmatrix},
\]
where each $B_i \in \mathbb{Z}^{m_i \times m_i}$ corresponds to the $i$-th SCC and $\sum_{i=1}^r m_i = N$.

The generator-wise initial system can be written as:
\[
\mathbf{x}_i^{B_i} = \mathbf{c}_i \cdot \prod_{j > i} \mathbf{x}_j^{-A_{ij}}, \qquad i = 1, \dots, r,
\]
where $\mathbf{x}_i$ denotes the variables in the $i$-th SCC.

We solve recursively in backward order ($i = r, r-1, \dots, 1$):
\begin{itemize}
    \item \textbf{Base case ($i=r$)}: The last block gives $\mathbf{x}_r^{B_r} = \mathbf{c}_r$.
          Since $B_r \in GL_{m_r}(\mathbb{Z})$ by assumption (ii), the monomial map
          $\phi_{B_r}: (\mathbb{C}^*)^{m_r} \to (\mathbb{C}^*)^{m_r}$ is an isomorphism.
          Hence there exists a unique solution $\mathbf{x}_r^* \in (\mathbb{C}^*)^{m_r}$.

    \item \textbf{Inductive step}: Suppose $\mathbf{x}_r^*, \dots, \mathbf{x}_{i+1}^*$ are uniquely
          determined. The $i$-th block equation becomes:
          \[
          \mathbf{x}_i^{B_i} = \mathbf{c}_i' := \mathbf{c}_i \cdot \prod_{j > i} (\mathbf{x}_j^*)^{-A_{ij}}.
          \]
          Since $\mathbf{c}_i \in (\mathbb{C}^*)^{m_i}$ and $\mathbf{x}_j^* \in (\mathbb{C}^*)^{m_j}$,
          we have $\mathbf{c}_i' \in (\mathbb{C}^*)^{m_i}$. Again, $B_i \in GL_{m_i}(\mathbb{Z})$
          implies $\phi_{B_i}$ is an isomorphism, so $\mathbf{x}_i^*$ is uniquely determined.
\end{itemize}

This proves part (a).

For part (b), let $x(t) \in V(I)$ with $\operatorname{val}(x(t)) = v$.
Write $x_k(t) = t^{v_k}(c_k + \text{higher order terms})$ for each $k$.
For any $f \in I$, substituting $x(t)$ into $f$ and extracting the lowest
$t$-order term gives $\operatorname{in}_v(f)(c) = 0$, where $c = \operatorname{in}(x(t))$.
In particular, this holds for each generator $f_i$, so $c \in V(J_v)$.
Since $V(J_v) = \{\mathbf{c}^*\}$ by part (a), we conclude
$\operatorname{in}(x(t)) = \mathbf{c}^*$.
\end{proof}

\begin{remark}[Scope and limitations]\label{rem:rigidity-scope}
Proposition~\ref{prop:initial-rigidity} establishes leading coefficient
rigidity: all Puiseux branches with the same valuation $v$ share the same initial
coefficient vector. This does not imply that there is only one Puiseux branch
with valuation $v$.

The number of Puiseux branches lifting a given tropical point $v$ is governed by
the tropical intersection multiplicity $\operatorname{mult}_{\mathrm{trop}}(v)$ and
the lifting theory (see Question~\ref{q:lifting-multiplicity} in Section~\ref{sec:outlook}).
In the unimodular case, the generator-wise initial system has a unique solution, but there may still
be multiple Puiseux branches distinguished by higher-order terms.

Conversely, if some diagonal block $B_i$ is not unimodular (i.e., $|\det(B_i)| > 1$),
then the corresponding generator-wise initial block has $|\det(B_i)| > 1$ solutions, and different Puiseux branches
may have different leading coefficients.
\end{remark}

\subsection{Examples of Binomial Rigidity and Torsion}
\begin{example}[A unimodular binomial block]
Consider the binomial initial system
\[
x_1-c_1=0,\qquad x_2-c_2x_1=0
\]
in \((\mathbb C^*)^2\), where \(c_1,c_2\in\mathbb C^*\).
The exponent-difference matrix is
\[
B=
\begin{pmatrix}
1 & 0\\
-1 & 1
\end{pmatrix},
\qquad \det(B)=1.
\]
Thus the block is unimodular.  By Theorem~\ref{thm:block-classification-rigorous},
the initial system has a unique torus solution,
\[
x_1=c_1,\qquad x_2=c_1c_2.
\]
Consequently, in a cell whose generator-wise initial system contains this
block, all Puiseux branches with that valuation have the same leading
coefficient on this block.
\end{example}
\begin{example}[A non-unimodular torsion block]
Consider
\[
x_1x_2-1=0,\qquad x_1-x_2=0
\]
in \((\mathbb C^*)^2\).  Equivalently,
\[
x_1x_2=1,\qquad x_1x_2^{-1}=1.
\]
The exponent-difference matrix is
\[
B=
\begin{pmatrix}
1 & 1\\
1 & -1
\end{pmatrix},
\qquad |\det(B)|=2.
\]
Hence the block has two torus solutions:
\[
(x_1,x_2)=(1,1),\qquad (-1,-1).
\]
This illustrates the torsion case of Theorem~\ref{thm:block-classification-rigorous}:
the initial fiber is reduced but has two distinct leading coefficient vectors.
\end{example}
\begin{remark}[Reduced torsion versus collision multiplicity]
The preceding examples show that binomial initial systems can produce either
a unique leading coefficient vector, in the unimodular case, or finitely many
distinct leading coefficient vectors, in the non-unimodular torsion case.  In
both cases the initial fiber is reduced for generic coefficients.

The valuation coalescence studied in the next section is different.  There,
several Puiseux branches have not only the same valuation vector but also the
same leading coefficient vector.  This stronger form of coalescence is not
captured by a reduced binomial block; it arises from a singular, nonreduced
initial fiber.
\end{remark}

\section{Cross-Prism Collision Degeneration}
\label{sec:rigidity}

The preceding sections describe the generic binomial regime of tropical
degeneration.  In that regime, a valuation vector lying in the relative
interior of a maximal cell selects two dominant monomials in each generator,
and the resulting initial system is controlled by an exponent-difference
matrix.  Unimodular blocks give unique leading coefficients, while
non-unimodular blocks give finite torsion fibers.

The purpose of this section is different.  We study a boundary degeneration
in which several monomials become simultaneously dominant and the initial
fiber is no longer reduced.  This is precisely the regime in which valuation
coalescence can occur: several Puiseux equilibrium branches may have the same
valuation vector and the same leading coefficient vector, while being
distinguished only by higher-order Puiseux terms.

Thus the cross-prism family below should not be viewed as another instance of
the generic binomial theorem.  Rather, it is a singular collision model showing
how an exponential number of algebraic branches can collapse into the
scheme-theoretic length of a single initial torus point.

\subsection{The Cross-Prism Representative}

Let \(\Pi_{\mathrm{cross}}^{(L)}\) denote the length-\(L\) two-layer cross-prism family, with \(L\geq 2\).  The variables are
\[
x_{1,A},\ldots,x_{L,A},
\qquad
x_{1,D},\ldots,x_{L,D},
\]
and all indices are taken modulo \(L\).  Thus \(x_{L+1,A}=x_{1,A}\),
\(x_{L+1,D}=x_{1,D}\), and \(\lambda_{L+1}=\lambda_1\).

The two layers \(A\) and \(D\) are arranged as directed \(L\)-cycles.  The
equation indexed by the \(k\)-th rung depends on the successor pair
\[
x_{k+1,A},\qquad x_{k+1,D}.
\]
Thus the indifference equations at rung \(k\) do not depend on the mixed
strategies \(x_{k,A}\) and \(x_{k,D}\) of the same rung.

We introduce a perturbation parameter \(t\) and work over the Puiseux field
\(K=\mathbb C\{\!\{t\}\!\}\).  Fix
\[
\beta\in \mathbb Q_{>0},
\qquad
\lambda_1,\ldots,\lambda_L\in \mathbb C^* .
\]
The collision-normalized cross-prism representative is defined by the
following \(2L\) equations:
\begin{align}
g_{k,A}
&=
(x_{k+1,A}-1)(x_{k+1,D}-1)-\lambda_{k+1}t^\beta,
\label{eq:collision-cross-prism-A}
\\
g_{k,D}
&=
x_{k+1,A}-x_{k+1,D},
\label{eq:collision-cross-prism-D}
\end{align}
for \(k=1,\ldots,L\).

Equivalently,
\[
g_{k,A}
=
x_{k+1,A}x_{k+1,D}
-
x_{k+1,A}
-
x_{k+1,D}
+
1
-
\lambda_{k+1}t^\beta.
\]
Hence each equation is strictly multilinear in the neighboring
mixed-strategy variables.

Every multilinear polynomial in the mixed strategies of a prescribed
neighborhood can be realized as the payoff-difference polynomial of a
binary-action network game by assigning payoff differences on the pure
neighborhood profiles.  In
\eqref{eq:collision-cross-prism-A}--\eqref{eq:collision-cross-prism-D},
the equation indexed by rung \(k\) depends only on the successor pair
\((k+1,A)\) and \((k+1,D)\).  Hence, for \(L\geq2\), the displayed equations
do not involve the player's own mixed strategy.  Consequently, the
collision-normalized system is a representative binary-action network game
with cross-prism dependency support, with Puiseux coefficients encoding the
degenerating payoff parameters.

\subsection{The Local Collision Block}

The essential mechanism is already visible in one rung.

\begin{lemma}[Local collision block]
\label{lem:local-collision-block}
Let \(\lambda\in\mathbb C^*\) and \(\beta\in\mathbb Q_{>0}\).  Consider the
system
\[
(a-1)(d-1)-\lambda t^\beta=0,
\qquad
a-d=0
\]
in \((K^*)^2\).  Then the system has two Puiseux solutions
\[
a(t)=d(t)
=
1+\varepsilon\sqrt{\lambda}\,t^{\beta/2},
\qquad
\varepsilon\in\{+1,-1\}.
\]
Both solutions have valuation vector \((0,0)\) and leading coefficient vector
\((1,1)\).

Moreover, at \(t=0\), the special fiber is supported at the single torus point
\[
(a,d)=(1,1),
\]
with scheme-theoretic length \(2\).
\end{lemma}

\begin{proof}
The equation \(a-d=0\) gives \(a=d\).  Substituting this into the first
equation yields
\[
(a-1)^2=\lambda t^\beta.
\]
Since \(K\) is algebraically closed, this equation has the two Puiseux
solutions
\[
a=d
=
1+\varepsilon\sqrt{\lambda}\,t^{\beta/2},
\qquad
\varepsilon\in\{+1,-1\}.
\]
Each branch has nonzero constant term \(1\).  Therefore both branches have
valuation vector \((0,0)\) and leading coefficient vector \((1,1)\).

For the special fiber, set
\[
u=a-1,\qquad w=d-1.
\]
At \(t=0\), the local ideal at \((a,d)=(1,1)\) is
\[
\langle uw,\ u-w\rangle.
\]
Since \(u=w\), the quotient is isomorphic to
\[
\mathbb C[u]/(u^2),
\]
which has length \(2\).  Hence the two Puiseux branches collide into a single
nonreduced torus point of length \(2\).
\end{proof}

\subsection{The \texorpdfstring{\(L\)}{L}-Rung Collision Family}

We now pass from one rung to the full length-\(L\) cross-prism family.

\begin{proposition}[Cross-prism collision and valuation coalescence]
\label{prop:cross-prism-collision}
For the successor-indexed collision-normalized cross-prism family
\[
g_{k,A}
=
(x_{k+1,A}-1)(x_{k+1,D}-1)-\lambda_{k+1}t^\beta,
\qquad
g_{k,D}
=
x_{k+1,A}-x_{k+1,D},
\]
for \(k=1,\ldots,L\), the generic fiber has \(2^L\) Puiseux torus branches.
They are given by
\[
x_{k,A}(t)=x_{k,D}(t)
=
1+\varepsilon_k\sqrt{\lambda_k}\,t^{\beta/2},
\qquad
\varepsilon_k\in\{+1,-1\},
\]
for \(k=1,\ldots,L\).

All these branches have the same valuation vector
\[
v=0
\]
and the same leading coefficient vector
\[
\mathbf c=(1,\ldots,1).
\]
Moreover, the initial fiber at \(v=0\) is supported at the unique torus point
\[
(1,\ldots,1)\in(\mathbb C^*)^{2L},
\]
with scheme-theoretic length \(2^L\).  In particular,
\[
m_{\mathrm{cl}}(0)=2^L.
\]
\end{proposition}

\begin{proof}
Although the equations are indexed by the source rung \(k\), each pair of
variables \(x_{j,A},x_{j,D}\) appears in the two equations indexed by the
predecessor rung \(j-1\).  Reindexing \(j=k+1\), the system decomposes into
\(L\) local collision blocks
\[
(x_{j,A}-1)(x_{j,D}-1)-\lambda_j t^\beta=0,
\qquad
x_{j,A}-x_{j,D}=0,
\]
for \(j=1,\ldots,L\).

By Lemma~\ref{lem:local-collision-block}, the \(j\)-th block has two Puiseux
branches
\[
x_{j,A}(t)=x_{j,D}(t)
=
1+\varepsilon_j\sqrt{\lambda_j}\,t^{\beta/2},
\qquad
\varepsilon_j\in\{+1,-1\}.
\]
The sign choices are independent for \(j=1,\ldots,L\), so the full system has
\(2^L\) Puiseux torus branches.

Since every coordinate of every branch has constant term \(1\), all branches
have valuation vector \(0\) and leading coefficient vector
\((1,\ldots,1)\).

It remains to compute the length of the limiting initial fiber.  At \(t=0\),
set
\[
u_j=x_{j,A}-1,
\qquad
w_j=x_{j,D}-1.
\]
The local ideal at the torus point \((1,\ldots,1)\) is
\[
\left\langle
u_jw_j,\ u_j-w_j
:
j=1,\ldots,L
\right\rangle.
\]
For each \(j\), the quotient by \(\langle u_jw_j,\ u_j-w_j\rangle\) is
isomorphic to \(\mathbb C[u_j]/(u_j^2)\), which has length \(2\).  The full
local quotient is the tensor product of these \(L\) length-two local algebras.
Therefore its length is
\[
2^L.
\]
Thus the \(2^L\) Puiseux branches collapse to a single nonreduced initial
torus point of length \(2^L\), proving the claimed valuation coalescence.
\end{proof}

In this example, the branch-theoretic cluster multiplicity
\(m_{\mathrm{cl}}(0)\) agrees with the scheme-theoretic length of the
special fiber at the unique limiting torus point.
Set-theoretically, the initial fiber consists of only one point, namely
\((1,\ldots,1)\).  Scheme-theoretically, its local algebra has length
\(2^L\).  Thus the \(2^L\) branches counted by \(m_{\mathrm{cl}}(0)\) are
recorded in the nilpotent structure of the special fiber.

\subsection{Comparison with the Generic Binomial Regime}

The collision-normalized cross-prism family is deliberately different from
the generic binomial regime studied in Sections~\ref{sec:initial}--\ref{sec:binomial-classification}.
For the collision family, the \(v=0\) initial forms are
\[
\operatorname{in}_0(g_{k,A})
=
(x_{k+1,A}-1)(x_{k+1,D}-1),
\qquad
\operatorname{in}_0(g_{k,D})
=
x_{k+1,A}-x_{k+1,D}.
\]
After the cyclic reindexing \(j=k+1\), these become the local collision
blocks
\[
(x_{j,A}-1)(x_{j,D}-1),
\qquad
x_{j,A}-x_{j,D}.
\]
The first initial form has four monomial terms after expansion:
\[
x_{j,A}x_{j,D}-x_{j,A}-x_{j,D}+1.
\]
Thus the initial system is not a generator-wise binomial system.  It lies on a
singular collision cell rather than in the relative interior of a generic
binomial maximal cell.

This distinction is essential.  In a generic unimodular binomial cell, the
initial torus point is reduced and Hensel-type lifting gives a unique local
Puiseux branch with the prescribed leading coefficient.  In a non-unimodular
binomial cell, several reduced torus points may occur, corresponding to
torsion in the exponent-difference lattice.  By contrast, the collision
cross-prism family has only one initial torus point, but that point is
nonreduced.  Consequently, several Puiseux branches share not only the same
valuation vector but also the same leading coefficient vector.

\begin{remark}[Reduced torsion versus collision multiplicity]
There are two distinct ways in which multiplicity can appear in a tropical
degeneration of a network game.  A binomial initial system may have several
reduced torus solutions, as measured by the torsion of an exponent-difference
lattice.  In that case, the branches share a valuation vector but have
different leading coefficient vectors.  The collision mechanism in this
section is stronger: the branches share both the valuation vector and the
leading coefficient vector, and their multiplicity is carried by the
nonreduced structure of the initial fiber.
\end{remark}

\subsection{Implication for Cross-Prism Equilibria}

The collision-normalized cross-prism family therefore realizes the strongest
form of valuation coalescence considered in this paper.  The generic fiber
contains exponentially many Puiseux equilibrium branches:
\[
2^L.
\]
As \(t\to0\), all of these branches satisfy
\[
\operatorname{val}(x(t))=0,
\qquad
\operatorname{lc}(x(t))=(1,\ldots,1).
\]
Hence the limiting game does not merely reduce the number of visible
leading-order configurations.  It hides exponentially many nearby
equilibrium branches inside a single nonreduced leading-order torus point.

In this sense, the cross-prism collision degeneration gives a valuation-level
refinement of the algebraic-degree count.  The global count \(2^L\) is not
lost in the tropical limit; it reappears as the scheme-theoretic length of
the initial fiber at the unique leading-order equilibrium configuration.

\section{Relation to Algebraic Degree and Counting}
\label{sec:degree}

This section explains how the degeneration theory developed here relates to the
algebraic-degree theory of network games in \cite{HuWangWang2026a}. The point is
not that the two papers compute the same invariant. Rather, they study two
different layers of the same algebraic object. The algebraic-degree theory counts
the number of isolated complex equilibria of a generic system in the algebraic
torus. The present paper studies how those equilibrium branches are organized
when the system is placed in a degenerating Puiseux family.

\subsection{Global algebraic degree}

Let $G=(V,E)$ be a player-dependency graph and let $M$ be Datta's structure
matrix for the corresponding sparse network game. For generic payoff
coefficients, the algebraic degree is
\begin{equation}
\mathcal{D}(G)=\operatorname{perm}(M).
\end{equation}
Equivalently, this degree is computed by admissible cycle-cover data in the
polynomial dependency graph. The authors' earlier work gives a tropical
interpretation of this identity by identifying the permanent with a stable
intersection count of tropical hypersurfaces associated with the indifference
equations \cite{HuWangWang2026a}.

The same global theory also yields structural laws for graph composition. In
particular, if the dependency graph decomposes into strongly connected
components
\[
G=C_1\sqcup\cdots\sqcup C_k,
\]
then
\begin{equation}
\mathcal{D}(G)=\prod_{r=1}^k \mathcal{D}(C_r).
\end{equation}
Thus the algebraic degree is a global counting invariant assembled from
component-level combinatorial data.

\subsection{Cellwise degeneration data}

The present paper starts from a valuation vector
\[
v\in \operatorname{Trop}(V(I))
\]
and studies the binomial generator-wise initial system
\[
J_v=\langle\operatorname{in}_v(f_1),\dots,\operatorname{in}_v(f_N)\rangle.
\]
\sloppy
Sections \ref{sec:initial}--\ref{sec:combinatorial} provide the hypotheses used below.
Under these hypotheses, the initial system is organized by its exponent-difference graph.
After ordering the strongly connected components,
the exponent-difference matrix becomes block upper triangular.
The binomial system can then be solved recursively along the induced partial order.
\fussy

The algebraic complexity of each block is measured by lattice data,
equivalently by the Smith normal form of the corresponding
exponent-difference matrix, rather than by a permanent.

This gives a cellwise invariant of a different nature. For a fixed tropical cell,
a binomial block contributes:
\begin{itemize}
\item a unique torus point in the unimodular case;
\item finitely many torus points in the full-rank non-unimodular case, counted by
      the relevant lattice index;
\item a positive-dimensional torus torsor in the rank-deficient case.
\end{itemize}
Hence the global permanent of the original system and the lattice indices of
binomial initial systems should be viewed as complementary invariants: the first
counts generic equilibria, while the second describes the leading-order
structure of branches inside a fixed tropical degeneration.

\subsection{A structural parallel}

\begin{proposition}[SCC-organized parallel between counting and degeneration]
\label{prop:scc-parallel}
Let $G$ be a network-game dependency graph with strongly connected component
decomposition
\[
G=C_1\sqcup\cdots\sqcup C_k.
\]
Then the following two principles are parallel.

\begin{enumerate}
\item The algebraic degree of a generic network game is multiplicative over the
strongly connected components:
\[
\mathcal{D}(G)=\prod_{r=1}^k \mathcal{D}(C_r).
\]

\item For a fixed tropical cell or valuation regime on which the generator-wise initial system is binomial, the
exponent-difference matrix decomposes along the strongly connected components
of the associated initial graph, and the solution structure of the generator-wise initial
system is controlled by the corresponding diagonal blocks.
\end{enumerate}

Consequently, both the global counting theory and the cellwise degeneration
theory organize algebraic complexity through strongly connected components, but
they attach different local invariants to the blocks: permanents in the former
case, and lattice indices or ranks in the latter.
\end{proposition}

\begin{proof}
The first assertion is the SCC multiplicativity theorem for the algebraic degree
of network games. The second assertion follows from the block-triangular
decomposition and recursive solution theory for binomial initial systems
established in Section~\ref{sec:binomial-classification}. The final statement is
a synthesis of these two facts.
\end{proof}

\begin{remark}
Proposition~\ref{prop:scc-parallel} is structural rather than numerical. A
permanent and a lattice index belong to different regimes: the permanent counts
isolated solutions of the original generic system, whereas the lattice index
controls the torus solutions of a specific binomial initial degeneration. The
common feature is the SCC-based mechanism by which algebraic complexity is
organized.
\end{remark}

\subsection{A cell-level comparison principle}

Although there is no universal equality between permanent data and determinant
data, one has a useful local comparison for a fixed binomial block.

\begin{proposition}[Cell-level comparison]
\label{prop:cell-level-comparison}
Let $B$ be a full-rank exponent-difference block arising from a binomial initial
system on a fixed tropical cell, and let $M_B$ be the $(0,1)$-support matrix of
$B$, defined by
\[
(M_B)_{ij}=1 \iff B_{ij}\neq 0.
\]
Then
\begin{equation}
|\det(B)|\leq \operatorname{perm}(M_B).
\end{equation}
In particular, the number of torus points of the corresponding block initial
system is bounded by the number of support permutations admitted by the same
block.
\end{proposition}

\begin{proof}
Expanding the determinant gives
\[
\det(B)=\sum_{\sigma\in S_m}\operatorname{sgn}(\sigma)
\prod_{i=1}^m B_{i,\sigma(i)}.
\]
Since the entries of an exponent-difference matrix lie in $\{-1,0,1\}$, each
nonzero term has absolute value $1$. Therefore
\[
|\det(B)|
\leq
\sum_{\sigma\in S_m}\left|\prod_{i=1}^m B_{i,\sigma(i)}\right|.
\]
The product on the right is nonzero precisely when every chosen entry is
supported by $M_B$. Hence
\[
|\det(B)|
\leq
\sum_{\sigma\in S_m}\prod_{i=1}^m (M_B)_{i,\sigma(i)}
=
\operatorname{perm}(M_B).
\]
\end{proof}

\begin{remark}
The proposition compares a determinant only with the permanent of its own support
matrix. It does not identify a tropical-cell-dependent binomial block with the
global Datta matrix of the original game. It should therefore be read as a local
counting-versus-cancellation principle: the permanent counts all
support-compatible configurations, whereas the determinant records the signed
surviving configurations of a binomial block.
\end{remark}

\subsection{Cycle rigidity}

A particularly clean parallel occurs for cycle-type structures.

\begin{corollary}[Cycle rigidity in the binomial regime]
\label{cor:cycle-rigidity-correspondence}
Suppose a binomial block is monomial-normalized to the form
\[
x_{\sigma(1)}=c_1,\quad \dots,\quad x_{\sigma(m)}=c_m
\]
for some permutation $\sigma\in S_m$. Then the corresponding exponent-difference
block is a permutation matrix, hence unimodular, and the block admits a unique
torus solution.
\end{corollary}

\begin{proof}
This is exactly Corollary~\ref{cor:normalized-unimodular}.
\end{proof}

\begin{remark}[Comparison with the algebraic-degree theory]
In the algebraic-degree theory, directed cycles are rigid from the global
counting perspective:
\[
\mathcal{D}(C_N)=1.
\]
In the present paper, monomial-normalized cycle blocks are rigid from the
degeneration perspective: they admit a unique torus point in the corresponding
initial system. Thus cycle-type structures suppress algebraic branching in both
regimes, although the relevant invariants are different.
\end{remark}

\subsection{The cross-prism collision family as a bridge}

The collision-normalized cross-prism family provides a concrete bridge
between global branch counts and local degeneration data.  In this family,
the generic fiber contains \(2^L\) Puiseux torus branches, while the limiting
initial fiber is supported at a single nonreduced torus point.

\begin{corollary}
\label{cor:cross-prism-bridge}
For the collision-normalized cross-prism family of
Proposition~\ref{prop:cross-prism-collision}, the following hold:
\begin{enumerate}
\item The generic fiber has \(2^L\) Puiseux torus branches.

\item All these branches have valuation vector \(v=0\) and leading
coefficient vector \((1,\ldots,1)\).

\item The initial fiber at \(v=0\) is supported at the unique torus point
\((1,\ldots,1)\), with scheme-theoretic length \(2^L\).
\end{enumerate}
Consequently, the branch count \(2^L\) is not lost in the tropical limit: it
is recorded as the length of a single nonreduced initial fiber.
\end{corollary}

\begin{proof}
All three assertions are exactly Proposition~\ref{prop:cross-prism-collision}.
\end{proof}

\begin{remark}
The collision-normalized representative has the same support-level
exponential count \(2^L\) as the cross-prism family in the algebraic-degree
theory.  The point of the example is not to rederive the general permanent
formula, but to show how an exponential number of branches can be absorbed
into a single nonreduced leading-order configuration under a singular
degeneration. The interpretation is
different: the algebraic-degree theory counts generic torus equilibria,
whereas the present degeneration tracks how these branches collide into a
single leading-order configuration.
\end{remark}

\subsection{Summary}

The relationship can be summarized as follows.  Algebraic-degree theory gives
a global count of generic torus equilibria, computed by permanents and
cycle-cover data.  Tropical degeneration theory refines this global count by
organizing degenerating branches cellwise, using binomial lattice data in
generic cells and nonreduced collision fibers in singular cells.

Thus the present paper does not duplicate the algebraic-degree theory. It
refines it asymptotically by explaining how the global equilibrium count is
distributed, separated, or collapsed across tropical valuation classes and
initial fibers in a degenerating Puiseux family.

\section{Discussion: Economic and Geometric Implications}
\label{sec:discussion}
The collision-normalized cross-prism family illustrates how a degenerating
network game can hide many nearby equilibrium branches inside a single
leading-order configuration.  The key phenomenon is not ordinary decoupling,
but collision: for \(t\neq0\), the system has \(2^L\) distinct Puiseux
branches, while at \(t=0\) these branches share the same valuation and leading
coefficient and are recorded by the nonreduced structure of the initial fiber.

\subsection{Near-Critical Strategic Responses}

Economically, the cross-prism topology provides an idealized model for
multiplex or layered systems in which two local response modes become nearly
indistinguishable under degeneration.  The perturbation term \(t^\beta\) can
be interpreted abstractly as a small separation parameter between local
strategic responses.  When \(t\neq0\), each rung admits two nearby response
branches.  As \(t\to0\), these branches collide at the same leading-order
mixed strategy.

Thus the limiting game does not merely simplify the system by discarding
higher-order perturbation terms.  Instead, it may conceal several nearby equilibria inside a single
visible leading-order equilibrium.  This gives an algebraic explanation for a
type of equilibrium fragility: small perturbations can split one apparent
limiting equilibrium into many distinct nearby branches.

\subsection{Valuation Coalescence and Nonreduced Initial Fibers}

Geometrically, Proposition~\ref{prop:cross-prism-collision} shows that the
collision-normalized cross-prism family realizes genuine valuation
coalescence.  All \(2^L\) Puiseux branches have
\[
\operatorname{val}(x(t))=0,
\qquad
\operatorname{lc}(x(t))=(1,\ldots,1).
\]
Set-theoretically, the limiting initial fiber contains only the torus point
\[
(1,\ldots,1).
\]
Scheme-theoretically, however, this point has length \(2^L\).  The missing
branches are therefore not visible as distinct leading coefficients; they
are encoded in the nilpotent structure of the special fiber.

This distinguishes collision multiplicity from torsion multiplicity in
binomial initial systems.  In the torsion case, one valuation class may
contain several reduced leading coefficient vectors.  In the collision case,
several Puiseux branches share the same leading coefficient vector, and the
multiplicity is carried by a nonreduced initial scheme.

\section{Outlook: Lifting, Fibers, and Singularity Questions}
\label{sec:outlook}

The framework developed in this paper opens several directions for future
research. We outline three main avenues below.

\begin{remark}[Proved results and open directions]\label{rem:proved-vs-conjectural}
The proved results of the paper are the generator-wise binomial reduction
(Theorem~\ref{thm:binomial_reduction}), binomial block classification
(Theorem~\ref{thm:block-classification-rigorous}), initial-coefficient
rigidity in unimodular binomial cells
(Proposition~\ref{prop:initial-rigidity}), and the collision-normalized
cross-prism coalescence result
(Proposition~\ref{prop:cross-prism-collision}).  The questions below are not
used in the proofs.  They are included only to indicate possible extensions
of the framework.
\end{remark}

\subsection{Lifting Multiplicity and Tropical Fibers}

A natural refinement of our framework concerns the \emph{lifting
multiplicity}: given a tropical valuation $v \in \operatorname{Trop}(V(I))$,
how many Puiseux branches lift to it? In the language of
Section~\ref{subsec:coalescence}, this asks for a precise formula for
the cluster multiplicity $m_{\mathrm{cl}}(v)$.

\begin{question}[Lifting multiplicity]\label{q:lifting-multiplicity}
Can the cluster multiplicity \(m_{\mathrm{cl}}(v)\) of a valuation class be
expressed in terms of tropical intersection multiplicities together with lattice
indices or graph-theoretic invariants of the generator-wise initial system
\(J_v\)?
\end{question}

A satisfactory answer would provide a local refinement of the algebraic degree
at the valuation level. The binomial block classification in this paper suggests
that lattice indices of SCC blocks should enter such a formula, but a general
formula is left for future work.

\subsection{Singularity Classification and ADE Types}

Another direction concerns the classification of degeneration singularities.
In the collision-normalized cross-prism family, each local rung has special
fiber
\[
\mathbb C[u]/(u^2),
\]
which is the local algebraic signature of a simple two-branch collision.
This suggests that singularity-theoretic methods may be useful for
classifying possible equilibrium collisions in multilinear network games.

It would be interesting to determine which singularity types can arise under
the multilinearity constraints of network-game indifference equations.  The
cross-prism degeneration studied here realizes a simple nonreduced collision,
but the appearance or exclusion of more complicated singularities requires a
separate local analysis.

\subsection{General Rigidity Criterion}
Finally, we return to the question of general cellwise rigidity.  The
binomial theory of Sections~\ref{sec:combinatorial}--\ref{sec:binomial-classification}
suggests that unimodularity is a robust sufficient condition for
initial-coefficient rigidity.  The collision example of
Section~\ref{sec:rigidity} shows that rigidity of the leading coefficient can
also occur in a different way: through a nonreduced initial fiber in which
several branches share the same leading coefficient.

\subsection{Summary}

These directions---lifting formulas, singularity classification,
and rigidity criteria---form a coherent research program that extends
the framework of this paper. We leave their systematic development to
future work.

\section*{Acknowledgments}

The authors thank Irem Portakal for bringing several related works in algebraic game theory to their attention.

\end{document}